\documentclass[12pt]{amsart}
\usepackage{amsmath,amsthm,amsfonts,amssymb}
\usepackage{color}
\usepackage{amsfonts,amsthm,amsmath,amssymb,verbatim}
\usepackage{graphicx}
\usepackage[active]{srcltx}
\usepackage[all,cmtip]{xy}
\usepackage{tikz}
\usepackage{xypic}
\usepackage{hyperref}
\hypersetup{
    colorlinks=true,
   linkcolor=blue,
   filecolor=magenta,
    urlcolor=cyan,
}

\newtheorem{thm}{Theorem}[section]

\newtheorem{lemma}[thm]{Lemma}

\theoremstyle{remark}

\newtheorem{example}[thm]{Example}
\newtheorem{remark}[thm]{Remark}
\newtheorem{defin}{Definition}

\newenvironment{dedication}
  {
   \thispagestyle{empty}
   \vspace*{\stretch{1}}
   \itshape             
   \raggedleft          
  }
  {\par 
   \vspace{\stretch{3}} 
  }


\def\C{\mathbb{C}}

\def\Q{\mathbb{Q}}
\def\Z{\mathbb{Z}}

\def\F{\mathbb{F}}

\def\G{\Gamma}

\def\x{R }
\def\xl{R}
\def\y{N }
\def\yl{N}

\title{Quadratic residue patterns and point counting on K3 surfaces}

\author[V. Kiritchenko]{Valentina Kiritchenko$^{1,2}$}
\address{$^1$HSE University,
Moscow, Russia}
\address{$^2$Institute for Information Transmission Problems, Moscow, Russia}
\address{$^3$Independent University of Moscow, Moscow, Russia}
\address{$^4$Aix Marseille Universit\'{e}, CNRS, Centrale Marseille, I2M UMR 7373, Marseille, France}
\address{$^5$University of California, Berkeley, USA}
\author[M. Tsfasman]{Mikhail Tsfasman$^{2,3}$}
\author[S. Vl\u{a}du\c{t}]{Serge Vl\u{a}du\c{t}$^{4}$}
\author[I. Zakharevich]{Ilya Zakharevich$^{5}$}

\email{vkiritch@hse.ru}
\email{MTsfasman@yandex.ru}
\email{serge.vladuts@univ-amu.fr}
\email{ilya@math.berkeley.edu}

\keywords{quadratic residues, elliptic curves, K3 surfaces}

\begin{document}
\maketitle
\begin{dedication}
To the memory of Lydia Goncharova.
\end{dedication}

\begin{abstract}Quadratic residue patterns modulo a prime are studied since 19th century. We state the last unpublished result of Lydia Goncharova, reformulate it and prior results in terms of algebraic geometry, and prove it. The core of this theorem is an unexpected relation between the number of points on a K3 surface and that on a CM elliptic curve.
\end{abstract}

\section{Introduction}                
Patterns formed by quadratic residues and nonresidues modulo a prime have been studied since the 19th century \cite{A,St,J} and continue to attract attention of contemporary mathematicians \cite{C,MT}.
Even the most elementary questions about distributions of quadratic (non)residues lead to difficult problems and deep results of number theory.
For instance, given a positive integer $\ell$ does there exist a prime number $p$ and an $\ell$-tuple of consecutive integers between $1$ and $p-1$ such that the $\ell$-tuple consists only of quadratic nonresidues (or residues) modulo $p$?
And how large is the least quadratic nonresidue modulo $p$?

In this note, we formulate a new result (Theorem \ref{t.main}) on patterns of quadratic (non)residues obtained by our late friend Lydia (Lida) Goncharova.
We recovered her result partly from her talk with the first author in December 2019, and partly from her older notes and e-mail correspondence.
We are grateful to David Kazhdan who helped us to recover Definition \ref{d.graph}.
Unfortunately, Lida had not written down a complete proof of her theorem, but we were able to  find a proof of her result.

As a byproduct, we get a K3 surface $X$ with the following property.
Point counting on $X$ over a finite field with $p$ elements reduces to point counting on an elliptic curve with complex multiplication.

In Section \ref{s.arithm}, we state Lida's result in the same way as she did it.
In Section \ref{s.geom}, we reformulate her result from a geometric viewpoint,  place it into a more general context and prove it.

\section{Quadratic residue patterns}\label{s.arithm}
Let $p$ be an odd prime.
Consider the sequence $1$, $2$, \ldots, $p-1$.
Replace every number $i$ in the sequence by the letter $\xl$ if $i$ is a quadratic residue modulo $p$, and by the letter $\yl$ otherwise.
Denote by $W_p$ the resulting word.
\begin{defin}
Let $S$ be a word of length $\ell\le{p-1}$ that contains only $\xl$ and $\yl$.
Define $n_p(S)$ as the number of subwords of $W_p$ that coincide with $S$, and are formed by $\ell$ consecutive elements of $W_p$.
\end{defin}
\begin{example}
Let $p=17$.
Then $W_p=\x\x\y\x\y\y\y\x\x\y\y\y\x\y\x\xl$.
If $\ell=3$, then $n_p(S)=2$ for all words $S$ of length $3$ except for $S=\x\x\xl$, while $n_p(\x\x\xl)=0$.
\end{example}

Note that if $\ell=1$, then $n_p(S)=\frac{p-1}2$.
Indeed, $n_p(\xl)=n_p(\yl)$ as the number of quadratic residues is equal to the number of nonresidues, and $n_p(\xl)+n_p(\yl)=p-1$.
For $\ell=2$, there are also explicit formulas that can be proved using elementary tricks \cite[Theorem 1]{A}.
For $\ell=3$, $4$, there are simple explicit formulas for certain linear combinations such as $n_p(\x\x\xl)+n_p(\y\y\yl)$ \cite{A, St}.
If $p=4k+3$ and $\ell=3$, then there are also simple formulas for $n_p(S)$ \cite[Part III, Formula I]{J}.
However, if $p=4k+1$ then the formula for $n_p(\x\x\xl)$ involves a more complicated ingredient.
Namely, define the function $J(k)$ by the formula
$$J(k)=\sum_{i=1}^{4k-2}\genfrac(){0.5pt}{0}{i(i+1)(i+2)}{p},$$
where $\genfrac(){0.5pt}{0}{a}{p}$ denotes the Legendre symbol.

If $p=4k+1$ and $\ell=3$, then all eight functions $n_p(S)$ can be easily expressed in terms of $J(k)$ \cite[Part III, Formula II]{J}.
This is closely related with Gauss' Last Entry (see \cite{M} for a short elementary exposition).
Jacobstahl also showed that $\frac{J^2(k)}{4}$ is one of the summands in the decompostion of $p$ into the sum of two squares \cite[Part II]{J}.
In particular, Jacobstahl's proof of formulas for $n_p(S)$ can also be regarded as the first proof of Gauss' Last Entry.
More geometrically, $n_p(S)$ can be computed via point counting on an elliptic curve with complex multiplication (CM).
We return to this approach in Example \ref{e.Gauss}.

If $\ell\ge 4$, the situation is even more complicated, for instance, a (not explicit) formula for $n_p(\x\x\x\xl)$ involves point counting on non-CM elliptic curves (see Section \ref{s.geom}).
We now consider a different problem about quadruples of residues, which still admits a solution in terms of $J(k)$.

Let $A=(a_1, a_2, a_3, a_4)$ be a quadruple of pairwise distinct residues modulo $p$.
We no longer assume that they are consecutive.
If $p=4k+1$, we may assign a graph $\Gamma_A$ to $A$ by the following rule.
Consider four vertices $v_1$,\ldots, $v_4$, and connect $v_i$ and $v_j$ by an edge
if and only if the difference $a_i-a_j$ is a quadratic residue modulo $p$
(since $-1$ is a quadratic residue this condition on $a_i-a_j$ is symmetric in $i$ and $j$).
In what follows, we identify quadruples that can be obtained from each other by a permutation or an additive translation, that is, we do not distinguish between $(a_1, a_2, a_3, a_4)$ and $(a_{\sigma(1)}, a_{\sigma(2)}, a_{\sigma(3)}, a_{\sigma(4)})$, where $\sigma\in S_4$, and between
$(a_1, a_2, a_3, a_4)$ and $(a_1+a, a_2+a, a_3+a, a_4+a)$, where $a$ is a residue modulo $p$.

\begin{defin} \label{d.graph}
Let $\Gamma$ be a graph with four vertices.
Define $n_p(\Gamma)$ as the number of all such quadruples $A$  that $\Gamma_A$ is topologically equivalent to $\Gamma$.
\end{defin}

As there are $11$ topological types of simple graphs with four vertices, there are $11$ numbers $n_p(\Gamma)$ for every $p$.

\begin{thm}[Goncharova] \label{t.main} Assume that $p=4k+1$.
All functions $n_p(\Gamma)$ (considered as functions of $k$) can be explicitly expressed as polynomials in $k$ and $d(k)$ where
$$d(k)=\frac{J(k)^2-4}{32}.$$

In particular, if $\Gamma$ is a complete graph with four vertices, then
$$n_p(\Gamma)=\frac{k(k-1)(k-4)+2kd(k)}{24}. \eqno(1)$$
\end{thm}

Goncharova had an elementary proof of this theorem, however, we were unable to recover the proof from her notes.
It seems that results of a similar flavor are proved in \cite{BE}.

Using computer we checked a version of formula (1) for all odd primes $p<20000$ (see formula (2) in Section 3 below. 
In the next section, we show that $n_p(\Gamma)$ can be computed by counting points on a K3 surface.

\section{Number of points on elliptic curves and K3 surfaces}\label{s.geom}
We now explain geometric meaning of numbers $n_p(S)$ and $n_p(\Gamma)$.
For simplicity, we focus on the cases where
$$S= \underbrace{\x\x\ldots\xl}_{\ell}$$
and $\Gamma$ is a complete graph on $\ell$ vertices (the other cases are completely analogous).

Let $\F_p$ be the finite field with $p$ elements.
Consider a curve $C\subset\bar\F_p^\ell$ obtained by intersecting $\ell-1$ quadrics:
$$x_2^2-x_1^2=1, \quad x_3^2-x_2^2=1, \ldots, \quad x_\ell^2-x_{\ell-1}^2=1.$$
Then every point $(x_1,\ldots,x_\ell)\in C$ with non-zero coordinates produces the $\ell$-tuple
$(x_1^2,x_2^2,\ldots, x_\ell^2)$ of consecutive quadratic residues modulo $p$.
Taking into account that $x_i$ and $-x_i$ produce the same quadratic residue $x_i^2$ we get that
$$n_p(S)=\frac{\# C^\circ(\F_p)}{2^\ell},$$
where $C^\circ\subset C$ consists of points whose coordinates are non-zero.
By the adjunction formula, the genus of $C$ is equal to $2^{\ell-2}(\ell-3)+1$.
In particular, $C$ has genus $0$ for $\ell=2$.
Hence, there is a simple explicit formula for  $n_p(\x\xl)$.

\begin{example}\label{e.Gauss} If $\ell=3$, then $C$ is the intersection of two quadrics in a 3-dimensional affine space.
The genus of $C$ is $1$, that is, the projective closure of $C$ is an elliptic curve.
Using rational parameterization of quadrics:
$$x_2+x_1=u, \quad x_2-x_1=u^{-1}; \quad x_3+x_2=v, \quad x_3-x_2=v^{-1},$$
we get that the projective closure of $C$ is isomorphic to the projective closure $E$ of the plane curve given by equation:
$$u(v^2-1)=v(u^2+1).$$
Since this curve has an automorphism of order $4$ over $\C$:
$$(u,v)\mapsto (iv, iu)$$
we get that $E$ is a CM elliptic curve, and ${\rm End}(E)=\Z[i]$.
Indeed, if we represent $E(\C)$ as the quotient of $\C$ by a lattice $L$, then $L$ has an automorphism of order 4.
Hence, a fundamental parallelogram of $L$ is a square, and $\C/L\simeq \C/\Z[i]$ (see also \cite[Example III.4.4]{S} for a more algebraic approach).

There is an isomorphism (over rational numbers) between $E$ and the elliptic curve with the Weierstrass form:
$$y^2=x^3-x.$$
The Edwards form \cite{E} of the same curve over $\C$ is 
$$x^2+y^2=1-x^2y^2.$$
Note that this is exactly the same equation as the one in Gauss' Last Entry.
It was historically studied in connection with lemniscate elliptic functions.
Nowadays, Edwards form is widely used in elliptic curve cryptography since it admits a more computationally efficient implementation of the addition law.

According to Gauss the number of points on the Edwards form of $E$ over $\F_p$ for $p=4k+1$ is equal to $(a_p-1)^2+b_p^2$ where $a_p+b_pi$ is a prime divisor of $p$ in the ring of Gaussian integers $\Z[i]$.
The prime divisor is uniquely characterized by the condition that $a_p-1+b_pi$ is divisible by $2+2i$.
This observation of Gauss was crucial for development of the theory of complex multiplication (see \cite[Section 3.1]{V} for a historical overview).
\end{example}
\begin{remark}
Using the formula for $n_p(\x\x\xl)$ from \cite[Part III]{J} we can express $J(k)$ (as defined in the previous section) by counting points on (the Weierstrass form of) $E$ over $\F_p$ (for $p=4k+1$) as follows:
$$J(k)=\# E(\F_p)-p-1.$$
Note also that $J(k)$ and $2a_p$ are equal up to a sign.
The precise relationship is determined in \cite[Part II]{J}:
$$2a_p=(-1)^{k+1}J(k).$$
In particular, the Edwards and Weierstrass forms of $E$ are not isomorphic over rational numbers.
\end{remark}
If $\ell=4$, then $C$ has genus $5$, however, a simple trick
\cite[proof of Theorem 2.1]{C} reduces computation of $n_p(S)$ to computation of
$(2^\ell-1)$ sums of the form:
$$\sum_{a\in\F_p}\genfrac(){0.5pt}{0}{f_I(a)}{p}$$
for all polynomials $f_I(a)=(a+i_1)\cdots(a_i+i_r)$, where $0\le i_1<\ldots<i_r\le \ell-1$.
Hence, it is enough to count points on all the curves $C_I\subset \F_p^2$ given by  equations $y^2=f_I(x)$ for $I=(i_1,\ldots,i_r)\subset (0,\ldots,l-1)$.
For $\ell\le4$ we get five curves of genus $1$ (the other curves have genus $0$):
$$\mbox{(a)} \; y^2=x(x+1)(x+2); \quad \mbox{(b)} \; y^2=x(x+1)(x+3); \quad \mbox{(c)} \; y^2=x(x+2)(x+3);$$
$$\mbox{(d)} \; y^2=(x+1)(x+2)(x+3);\quad \mbox{(e)} \; y^2=x(x+1)(x+2)(x+3).$$
Curves (a) and (d) are isomorphic over $\Q$, and curves
(b) and (c) are isomorphic over $\C$.
While (a) and (d) are CM curves, (b), (c) and (e) are not.
In particular, one needs a deep result from number theory, namely, the Sato--Tate conjecture (see \cite[Section C.21]{S}), in order to decribe the statistics of numbers $n_p(\x\x\x\xl)$ for varying primes $p$. 

 Let us briefly recall the corresponding results. Let $E/\Q$ be an elliptic curve, then for any prime $p$ with good reduction $E_p/\F_p$ modulo $p$ we have $\# E(\F_p)=p+1-a_p(E)$, $a_p(E), |a_p(E)|<2\sqrt p$ being the trace of $p$-Frobenius operator.
There are two cases: 

A. $E$  is a CM-curve, that is $\mathrm{End}(E)\ne\Z;$ B.  $\mathrm{End}(E)=\Z$, that is, $E$ has no complex multiplication.  

We can consider the normalized trace $t_E(p):=\frac{a_p(E)}{2\sqrt p}\in I:= ]-1,1[$. Its behaviour in the cases A and B is very different:

A. The sequence  $$T(E):=\{t_E(p): E\;\; \hbox{has a good reduction modulo } p\}=\{Primes\setminus S(E)\}$$
with a finite set $ S(E)$ of primes is uniformly distributed with respect to the Lebesgue measure $dt$ on $I$.

B. The sequence $T(E)$ is uniformly distributed with respect to the "semi-cirlce"  $d\mu$
 measure on $I$,

$$  d\mu(t):=\frac{2\sqrt{1-t^2}}{\pi}dt.$$

The result in the case A follows from class field theory, Hecke's classical results on the expression of the (global) zeta-function of $E$ via $L$-series with Gr\"ossencharakter and standard distribution results following from Ikehara's theorem, see the last chapter of Lang's "Number theory" \cite{La}.

The result in the case B is just the Sato-Tate conjecture  formulated around 1960 and proved completely in \cite{HSBT}.

Therefore, for $n_p(\x\x\x\xl)$ we get (under a very plausible condition of the independence of various terms in the formula):
$$n_p(\x\x\x\xl)=\frac{p-1}{16} +2\sqrt p \xi+o(\sqrt p)$$
for $\xi:= b \lambda +c\lambda_{ST}$ with rational positive $ a,b,c,  b+c=1-a$
where   $\lambda$ (resp.$\lambda_{ST}$ is a random variable equidistributed on $]-1,1[$ (resp. distributed following the ST measure).

If $\ell=5$, then the same approach yields a curve of genus $2$ (in addition to curves of smaller genera):
$$y^2=x(x+1)(x+2)(x+3)(x+4).$$
However, counting points on this curve can be reduced to counting points on elliptic curves.   
Indeed, this curve has a non-hyperelliptic involution $\sigma:  x+2\mapsto -(x+2), y\mapsto   iy$.  
 Hence, the statistics for $n_p(S)$ can still be computed in terms of the Sato--Tate distribution for elliptic curves.
If $\ell=6$, then there are ``real'' curves of genus $2$, which can not be simplified, that is, the statistics for
$n_p(S)$ would rely on a conjectural Sato--Tate distribution for curves of genus $2$ (see \cite[Theorem 4.8]{Su}) and can not be reduced to the case of elliptic curves.
For $\ell\ge7$, current tools of number theory do not yield even conjectural description of statistics for $n_p(S)$.

We now switch to numbers $n_p(\G)$.
To interpret $n_p(\G)$ geometrically define the variety $X_\Gamma\subset\F_p^\ell\times\F_p^d$, where $d:={\frac{\ell(\ell-1)}{2}}$.
The defining equations are
$$x_j-x_i=y_{i,j}^2, \quad 1\le i<j\le \ell,$$
where $(x_1,\ldots,x_\ell)$ are coordinates in $\F_p^\ell$, and $(y_{i,j})_{1\le i<j\le\ell}$ are coordinates in $\F_p^d$.
Consider the subvariety $X^\circ_\Gamma\subset X_\Gamma$ that consists of all points with pairwise distinct coordinates.
Since there are natural free actions of the permutation group $S_\ell$ and of the additive group $\F_p$ on $X^\circ$, we get that
$$n_p(\G)=\frac{\# X^\circ (\F_p)}{\ell!2^dp}.$$

Note that the multiplicative group $(\F_p^*)^2\subset \F_p^*$ also acts freely on $X^\circ_\Gamma(\F_p)$.
Taking quotient by actions of $\F_p$ and $(\F_p^*)^2$ we may count points on
$X^\circ_\Gamma(\F_p)$ by counting points on $Y_\Gamma(\F_p)$, where $Y_\Gamma\subset\F_p^\ell\times\F_p^d$ is given by equations:
$$x_1=0, \ x_\ell=1;  \quad x_j-x_i=y_{i,j}^2, \quad 1\le i<j\le \ell.$$
The case $\ell=2$ is trivial.
If $\ell=3$, then $Y_\Gamma$ is a pair of conics, so there is a simple formula for $n_p(\Gamma)$.

{\bf Proof of Theorem 2.2} Let $\ell=4$, then $Y_\Gamma$ consists of two copies of the K3 surface $S$ obtained as the intersection of three quadrics in $\F_p^5$:
$$y_{1,2}^2+y_{2,3}^2=y_{1,3}^2; \quad y_{2,3}^2+y_{3,4}^2=y_{2,4}^2; \quad y_{1,2}^2+y_{2,3}^2+y_{3,4}^2=1.$$
In particular, if $p=4k+1$, then formula (1) is equivalent to the following formula:
$$\#S(\F_p)=(p-1)^2+J(k)^2+4. \eqno (2)$$
\begin{lemma} Formula (2) is equivalent to the formula (5) below. \end{lemma}
\begin{proof} 
We use the following rational parameterization of the first two quadrics defining $S$:
$$y_{1,2}=(x^2-1)s, y_{2,3}=2xs, y_{1,3}=(x^2+1)s,$$  
$$y_{3,4}=(y^2-1)t, y_{2,3}=2yt, y_{2,4}=(y^2+1)t,$$
and if we  put $ z:=ys^{-1},s:=ytx^{-1}$ then $s=z/y, t=x/z,t/s=y/x$ and  the third equation reads
$$z^2=(x^2y^2+1)(x^2+y^2),$$
which defines a singular K3  surface $X$ in a $3$-space.
This gives  birational maps
$$\phi:X\longrightarrow S, \psi: S\longrightarrow X,$$
given by
$$y_{1,2}=(x^2-1)z/y,y_{2,3}=2xsz/y,y_{1,3}=(x^2+1)z/y $$
$$y_{3,4}=(y^2-1)x/z, y_{2,3}=2yx/z, y_{2,4}=(y^2+1)x/z,$$
for $\phi$ and 
$$ x=\frac{y_{2,3}}{y_{1,3}-y_{1,2} }, y=\frac{4y_{2,3}}{(y_{1,3}-y_{1,2})^2 (y_{2,3}-y_{3,4})},z=\frac{2y_{2,3}}{(y_{1,3}-y_{1,2}) (y_{2,3}-y_{3,4})}$$
for $\psi$. These maps establish a bijection 
between $X\setminus D$ and $S\setminus D_1$,
where the divisors $D,D_1$ are given by $\{y=z=0\}$ and  $\{y_{1,3}-y_{1,2}=y_{2,3}=0 \}$, respectively. A simple calculation shows that there are $6p-4$ $\F_p$-rational points on $D$ and $2p-1$ such points on $D_1$ which proves the desired equivalence.
\end{proof}

\begin{thm}
Consider the affine surface $X$ in a $3$-space given by the equation:
$$z^2=(x^2y^2+1)(x^2+y^2), \eqno(3)$$
and the affine plane curve $E$ given by the equation:
$$y^2=x^3-x. \eqno(4)$$
Let $M_p$ and $N_p$, respectively, denote the number of solutions of (3) and (4) modulo a prime $p$.
Then $M_p$ and $N_p$ are related by the following identity:
$$M_p=(p+1)^2+(N_p-p)^2+1.\eqno(5)$$
\end{thm}
\begin{proof}
Making change of variables $x=x_1$, $y=tx_1$, $z=y_1x_1$ we may replace the surface $X$ by the surface $X'$ given in coordinates $(x_1, y_1, t)$ by the equation:
$$y_1^2=(t^2x_1^4+1)(t^2+1). \eqno(6)$$
It is easy to check that 
$$\#X(\F_p)=\#X'(\F_p)+p.$$
We now count separately points on $X'\setminus X'_0$ and on $X'_0$ where $X'_0:= X'\cap (\{x_1=0\}\cup\{y_1=0\}\cup\{t=0\})$.

By inclusion--exclusion formula we get that if $p\equiv 1 \pmod 4$, then
$$\#X'_0(\F_p)=\#[X'\cap\{x_1=0\}](\F_p)\ + \ \#[X'\cap\{y_1=0\}](\F_p) \ + \ \#[X'\cap\{t=0\}](\F_p) \ -$$
$$- \ \#[X'\cap\{x_1=y_1=0\}](\F_p) \ - \ \#[X'\cap\{x_1=t=0\}](\F_p) \ =$$
$$= \ (p-1)+(4p-10)+2p-2-2 \ = \ 7p-15$$ 
Here we used that $X'\cap\{y_1=t=0\}$ is empty.

To count points on $X'\setminus X'_0$ we regard $t$ as a parameter.
If $t$ is fixed then equation (6) defines an elliptic curve $X_t\subset X'$. 
The number of points on such a curve is uniquely determined by the quadratic residue pattern formed by $(t, t^2+1)$. 
Hence, there are four cases to consider: $R R$, $R N$, $N R$, $N N$.
There are four elliptic curves $E_1$, $E_2$, $E_3$, $E_4$, respectively, that parameterize pairs $(t,t^2+1)$ in every case.
For instance, if $(t,t^2+1)$ has pattern $RR$ then $t=s^2$ and $t^2+1=u^2$ where
$(s,u)$ is a point with non-zero coordinates on the plane curve $E_1$ given by the equation $u^2=s^4+1$.
Clearly, the points $(\pm s,\pm u)$ and $(\mp s,\pm u)$ correspond to the same value of $t$.
Hence, the number of values $t$ such that the pair $(t,t^2+1)$ has 
pattern $RR$ is equal to $\frac{1}{4}\#E^\circ_1(\F_p)$ where $E^\circ_1$ denotes $E_1\setminus(\{s=0\}\cup\{u=0\})$.
An analogous equality holds in the other three cases.

It is interesting that for all values of $t$ parameterized by $E_i$ the corresponding elliptic curve $X_t$ is isomorphic to $E_i$ over $\F_p$.
For instance, if $t$ and $t^2+1$ are quadratic residues, then $X_t$ is isomorphic to the curve $y^2=x^4+1$ by the change of variables $x=sx_1$, $y=u^{-1}y_1$.
Hence, we get the following identity:
$$\#[X'\setminus X'_0](\F_p)=\frac{1}{4}\sum_{i=1}^{4}\#E^\circ_i(\F_p)^2$$

Using Table 1 we get that for $p=8k\pm1$ the right hand side is equal to
$$\frac{1}{4}[(p-7+a)^2+2(p-3-a)^2+(p+1+a)^2]=p^2-6p+17+a^2.$$
The computation with Table 2 for $p=8k\pm3$ yields the same answer.
It remains to note that the curves $E_1$ and $E$ have the same trace of Frobenius up to a sign.

\begin{table}[h]
		\centering
\begin{tabular}{|l|c|c|c|c|} \hline\hline
{\em Curve} & \# points at $\infty$ & \# points at $u=0$ or $s=0$ & sum & trace of Frobenius\\ \cline{2-5}
$u^2=s^4+1$ & 2 & 6 & 8 & $a$\\
\cline{2-5}
$\delta u^2=s^4+1$ & 0 & 4 & 4 & $-a$\\
\cline{2-5}
$u^2=\delta^2s^4+1$ & 2 & 2 & 4 & $-a$\\
\cline{2-5}
$\delta u^2=\delta^2s^4+1$ & 0 & 0 & 0 & $a$\\
\hline\hline
\end{tabular}
\medskip

\caption{Case $p=8k\pm1$}
\end{table}

Here and below $\delta\in\F_p^*$ denotes a quadratic non-residue.

\begin{table}[h]
		\centering
\begin{tabular}{|l|c|c|c|c|} \hline\hline
{\em Curve} & \# points at $\infty$ & \# points at $u=0$ or $s=0$ & sum & trace of Frobenius\\ \cline{2-5}
$u^2=s^4+1$ & 2 & 2 & 4 & $a$\\ 
\cline{2-5}
$\delta u^2=s^4+1$ & 0 & 0 & 0 & $-a$\\
\cline{2-5}
$u^2=\delta^2s^4+1$ & 2 & 6 & 8 & $-a$\\
\cline{2-5}
$\delta u^2=\delta^2s^4+1$ & 0 & 4 & 4 & $a$\\
\hline\hline
\end{tabular}
\medskip

\caption{Case $p=8k\pm3$}
\end{table}

\end{proof}

\end{document}